\documentclass[11pt]{article}
\usepackage[a4paper, margin = 3.4cm]{geometry}
\usepackage[english]{babel}
\usepackage{latexsym,amsmath,enumerate,graphics,enumerate,amsthm,tikz,hyperref,float}
\usepackage[mathscr]{euscript}
\usepackage[affil-it]{authblk}
\usepackage{enumerate,amsthm,dsfont,pstricks}
\usepackage{latexsym,amsmath,amssymb,amscd,wrapfig,graphicx}

\newtheorem{theorem}{Theorem}[section]
\newtheorem{lemma}[theorem]{Lemma}
\newtheorem{proposition}[theorem]{Proposition}
\newtheorem{corollary}[theorem]{Corollary}

\theoremstyle{definition}

\newtheorem{remark}[theorem]{Remark}
\theoremstyle{remark}

\let\phi=\varphi
\def\epsilon{\varepsilon}

\def\0{\mathbf{0}}
\def\bgamma{\bar{\gamma}}

\newcommand{\comment}[1]{}

\numberwithin{equation}{section}

\let\epsilon=\varepsilon

\makeatletter
\def\@maketitle{%
  \newpage
  \null
  \vskip 2em%
  \begin{center}%
  \let \footnote \thanks
    {\Large\bfseries \@title \par}%
    \vskip 1.5em%
    {\normalsize
      \lineskip .5em%
      \begin{tabular}[t]{c}%
        \@author
      \end{tabular}\par}%
    \vskip 1em%
    {\normalsize \@date}%
  \end{center}%
  \par
  \vskip 1.5em}
\makeatother

\begin{document}

\title{\sc \Large A metric version of Poincar\'e's theorem concerning biholomorphic inequivalence of domains}

\author{Bas Lemmens%
\thanks{Email: \texttt{B.Lemmens@kent.ac.uk}, The author gratefully acknowledges the support of the EPSRC (grant EP/R044228/1)}}
\affil{School of Mathematics, Statistics \& Actuarial Science,
University of Kent, \\Canterbury, CT2 7NX, United Kingdom}

\maketitle
\date{}

\begin{abstract}
We show that if $Y_j\subset \mathbb{C}^{n_j}$ is a bounded strongly convex domain with $C^3$-boundary for $j=1,\dots,q$, and $X_j\subset \mathbb{C}^{m_j}$ is a bounded convex domain for $j=1,\ldots,p$, then  the product domain $\prod_{j=1}^p X_j\subset \mathbb{C}^m$ cannot be isometrically embedded into  $\prod_{j=1}^q Y_j\subset \mathbb{C}^n$ under the Kobayashi distance, if $p>q$. This result generalises Poincar\'e's theorem which says that there is no biholomorphic map from the  polydisc onto the Euclidean ball in $\mathbb{C}^n$ for $n\geq 2$. 

The method of proof only relies on the metric geometry of the spaces and will be derived from a result  for products of proper geodesic metric spaces with the sup-metric. In fact, the main goal of the paper is to establish a general criterion, in terms of certain asymptotic geometric properties of the individual metric spaces, that yields an obstruction for the existence of an isometric embedding between product metric spaces.  
\end{abstract}

{\small {\bf Keywords:} Product metric spaces, Product domains, Kobayashi distance, isometric embeddings, metric compactification, Busemann points, detour distance}

{\small {\bf Subject Classification:} Primary 32F45; Secondary 51F99}


\section{Introduction}
Numerous theorems in several complex variables are instances of results in metric geometry. In this paper we shall see that a classic theorem due to Poincar\'e \cite{Po}, which  says that there is no biholomorphic map from  the polydisc $\Delta^n$  onto the  (open) Euclidean ball $B_n$ in $\mathbb{C}^n$ if $n\geq 2$, is a case in point.  In fact, it is known \cite{Ma,Zwo1,Zwo2} that there exists no surjective Kobayashi distance isometry of  $\Delta^n$ onto $B_n$ if $n\geq 2$. More generally one may wonder when it is possible to isometrically embed a product domain $\prod_{j=1}^p X_j\subset \mathbb{C}^m$ into another product domain $\prod_{j=1}^q Y_j\subset \mathbb{C}^n$ under the Kobayashi distance. 
In this paper we show, among other results,  the following theorem. 
\begin{theorem}\label{thm:poincare}
Suppose that $X_j\subset \mathbb{C}^{m_j}$ is a bounded convex domain for $j=1,\ldots,p$, and  $Y_j\subset \mathbb{C}^{n_j}$ is a bounded strongly convex domain with $C^3$-boundary for $j=1,\dots,q$.. If $p>q$, then there is no isometric embedding of  $\prod_{j=1}^p X_j$ into  $\prod_{j=1}^q Y_j$ under the Kobayashi distance.
\end{theorem}
Note that Poincar\'e's theorem is a special case where $p=n\geq 2$ and $q=1$, as the boundary of the Euclidean ball is smooth. The case where $\sum_j m_j = \sum_j n_j$ and the isometry is surjective was analysed by Zwonek \cite[Theorem 2.2.5]{Zwo1} who used different methods. 

A key property of the Kobayashi distance is the product property, see \cite[Theorem 3.1.9]{Ko}. Indeed, if $X_j\subset \mathbb{C}^{m_j}$ is a bounded convex domain for $j=1,\ldots,p$, then the Kobayashi distance, $k_X$, on the product domain $X:=\prod_{j=1}^p X_j$ satisfies
\[
k_X(w,z) =\max_{j=1,\ldots,p} k_{X_j}(w_j,z_j) \mbox{\quad  for all }w=(w_1,\ldots,w_p),z=(z_1,\ldots,z_p)\in X.
\]
In view of the product  property it natural to consider  product metric spaces with the sup-metric. Given metric spaces $(M_j,d_j)$, $j=1,\ldots,p$, the {\em product metric space} $(\prod_{j=1}^p M_j,d_\infty)$ is given by 
\[
 d_\infty(x,y) := \max_j d_j(x_j,y_j)\mbox{\quad for $ x=(x_1,\ldots,x_p),y=(y_1,\ldots,y_p)\in \prod_{j=1}^p M_j$.}
 \]
In this general context it is interesting to understand when one can isometrically embed a product metric space into another one.  The main goal of this paper is to establish a general criterion, in terms of certain asymptotic geometric properties of the individual metric spaces, that yields an obstruction for the existence of an isometric embedding between product metric spaces, and to show how this criterion can be used  to derive Theorem \ref{thm:poincare}. 

The key concepts from metric geometry involved are: the horofunction boundary of proper geodesic metric spaces,  almost geodesics, Busemann points, the detour distance, and the parts of the horofunction boundary, which will all be recalled in the next section. Our main result is the following.  
\begin{theorem}\label{thm:main} Suppose that $(M_j,d_j)$ is a proper geodesic metric space containing an almost geodesic sequence for $j=1,\ldots,p$, and $(N_j,\rho_j)$ is a proper geodesic metric space such that all its horofunctions are Busemann points, and $\delta(h_j,h_j')=\infty$ for all  $h_j\neq h_j'$ Busemann points  of $(N_j,\rho_j)$, for $j=1,\ldots,q$. If $p>q$, then there exists  no isometric embedding of $(\prod_{j=1}^p M_j,d_\infty)$ into $(\prod_{j=1}^qN_j, d_\infty)$. 
\end{theorem}
The assumptions that each horofunction is a Busemann point and that any two distinct Busemann points lie at infinite detour distance from each other is a type of regularity condition on the asymptotic geometry of the space, which is satisfied by numerous metric spaces, such as finite dimensional normed spaces with smooth norms \cite{Wa2}, Hilbert geometries on bounded strictly convex domains with $C^1$-boundary \cite{Wa1}, and, as we shall see in  Lemma \ref{detour2}, Kobayashi metric spaces $(D,k_D)$, where $D\subset \mathbb{C}^n$ is a bounded strongly convex domain with $C^3$-boundary. 

It turns out that the parts of the horofunction boundary and the detour distance in product metric spaces have a special structure that is closely linked to a quotient space of $(\mathbb{R}^n,2\|\cdot\|_\infty)$, where $\|x\|_\infty=\max_j |x_j|$. More precisely, if we let $\mathrm{Sp}(\mathbf{1}) :=\{\lambda (1,\ldots,1)\in\mathbb{R}^n\colon \lambda\in\mathbb{R}\}$, then the quotient space $\mathbb{R}^n/\mathrm{Sp}(\mathbf{1})$ with respect to $2\|\cdot\|_\infty$ has the {\em variation norm} as the quotient norm, which is given by 
\begin{equation}\label{var}
\|\overline{x}\|_{\mathrm{var}} :=\max_jx_j + \max_j (-x_j)\mbox{\quad for }\overline{x} \in \mathbb{R}^n/\mathrm{Sp}(\mathbf{1}),
\end{equation}
see \cite[Section 4]{LRW}.  It is known, e.g.,  \cite[Proposition 2.2.4]{LNBook}, that $(\mathbb{R}^n/\mathrm{Sp}(\mathbf{1}), \|\cdot\|_{\mathrm{var}})$ is isometric to the Hilbert metric space on the open $(n-1)$-dimensional simplex. 

We show in Theorem \ref{thm:parts} that if, for $j=1,\ldots,q$, we have that $(N_j,\rho_j)$ is a proper geodesic metric space such that all its horofunctions are Busemann points, and $\delta(h_j,h_j')=\infty$ for all $h_j\neq h_j'$  Busemann points  of $(N_j,\rho_j)$,  then each part of  $(\prod_{j=1}^qN_j, d_\infty)$ is  isometric to $(\mathbb{R}^n/\mathrm{Sp}(\mathbf{1}), \|\cdot\|_{\mathrm{var}})$ for some $1\leq n\leq q$. 

The horofunctions for the product of two metric spaces have been considered by Walsh \cite[Section 8]{Wa4}. Some of our results are extensions of his work to arbitrary finite products, and the ideas of some of the proofs are quite similar. For the reader's convenience we give full proofs and provide comments on the relation with the work in \cite{Wa4} where relevant.

The work in this paper has links to work by Bracci and Gaussier \cite{BG} who studied the interaction between  topological properties and  the metric geometry of hyperbolic complex spaces. It is also worth mentioning that various other aspects of the metric geometry of product metric spaces have been studied in  context of Teichm\"uller space in  \cite{DF, Mi}.

\section{The metric compactification of product spaces}
 In our set-up we will follow the terminology in \cite{Ka}, which contains further references and background on the metric compactification.
 
Let $(M,d)$ be a metric space, and let $\mathbb{R}^M$ be the space of all real functions on $M$ equipped with the  topology of pointwise convergence. Fix  $b\in M$, which is called the {\em basepoint}. Let $\mathrm{Lip}^1_b(M)$ denote the set of all functions $h\in\mathbb{R}^M$ such that $h(b)=0$ and $h$ is 1-Lipschitz, i.e., $|h(x)-h(y)|\leq d(x,y)$ for all $x,y\in M$. Then $\mathrm{Lip}^1_b(M)$ is a closed subset of $\mathbb{R}^M$. Moreover, as 
\[
|h(x)|= |h(x)-h(b)|\leq d(x,b)
\]
for all $h\in \mathrm{Lip}^1_b(M)$ and $x\in M$, we get that $\mathrm{Lip}_b^1(M)\subseteq [-d(x,b),d(x,b)]^M$, which is compact by Tychonoff's theorem. 
Thus, $\mathrm{Lip}^1_b(M)$ is a compact subset of $\mathbb{R}^M$. 

Now for $y\in M$ consider the real valued function 
\[
h_{y}(z) := d(z,y)-d(b,y)\mbox{\quad with $z\in M$.}
\] 
Then $h_y(b)=0$ and $|h_y(z)-h_y(w)| = |d(z,y)-d(w,y)|\leq d(z,w)$. Thus, $h_y\in  \mathrm{Lip}_b^1(M)$ for all $y\in M$. The closure of $\{h_y\colon y\in M\}$ is called the {\em metric compactification of $M$}, and is denoted $\overline{M}^h$. The boundary $\partial \overline{M}^h:= \overline{M}^h\setminus \{h_y\colon y\in M\}$ is called the {\em horofunction boundary of $M$}, and its elements are called {\em horofunctions}. Given a horofunction $h$ and $r\in\mathbb{R}$ the set $\mathcal{H}(h,r):=\{x\in M\colon h(x)<r\}$ is a called a  {\em horoball}.

We will assume that the metric space $(M,d)$ is {\em proper}, meaning that all closed balls are compact. Such metric spaces are separable, since every compact metric space is separable. It is known that if $(M,d)$ is separable, then the topology of pointwise convergence on $\mathrm{Lip}^1_b(M)$ is metrizable, and hence each horofunction is the limit of a sequence of functions $(h_{y^n})$ with $y^n\in M$ for all $n\geq 1$.  In general, however, horofunctions are limits of nets. 

A curve $\gamma\colon I\to (M,d)$, where $I$ is a possibly unbounded interval in $\mathbb{R}$, is called a {\em geodesic path} if 
\[
d(\gamma(s),\gamma(t))=|s-t|\mbox{\quad for all }s,t\in I.
\]
The metric space $(M,d)$ is said to be a {\em geodesic space} if for each $x,y\in M$ there exists a geodesic path $\gamma\colon [a,b]\to M$ with $\gamma(a) =x$ and $\gamma(b) =y$.  A proof of the following  well known fact can be found in \cite[Lemma 2.1]{LLN}. 
\begin{lemma}\label{lem1}
If $(M,d)$ is a proper geodesic metric space, then $h\in\partial \overline{M}^h$ if and only if there exists a sequence $(y^n)$ in $M$ such that $h_{y^n}\to h$ and $d(y^n,b)\to\infty$ as $n\to\infty$.
\end{lemma}
It should be noted that in the previous lemma it is necessary to assume that the metric space is proper. Indeed, consider the star graph with centre vertex  $b$ and edges $E_n = \{b,v^n\}$ of length $n$ for $n\in\mathbb{N}$. Then the sequence $(v^n)$ in the resulting path metric space, with basepoint $b$, satisfies 
\[\lim_{n\to\infty} h_{v^n}(x) = \lim_{n\to\infty}d(x,v^n)-d(b,v^n) = d(x,b) = h_b(x)\] 
for all $x$, and hence does not yield a horofunction.  

A sequence $(y^n)$  in $(M,d)$ is called an {\em almost geodesic sequence} if $d(y^n,y^0)\to\infty$ as $n\to\infty$, and for each $\epsilon >0$ there exists $N\geq 0$ such that 
\[
d(y^m,y^k) +d(y^k,y^0) - d(y^m,y^0) <\epsilon\mbox{\quad for all }m\geq k\geq N.
\]
The notion of an almost geodesic sequence  goes back to  Rieffel \cite{Ri} and was further developed in \cite{AGW,LW, Wa3,Wa2}. 
In particular, any almost geodesic sequence yields a horofunction, see \cite[Lemma 4.5]{Ri}. 
\begin{lemma} Let $(M,d)$ be a proper geodesic metric space. If $(y^n)$ is an almost geodesic sequence in $M$, then 
\[
h(x) = \lim_{n\to\infty} d(x,y^n)-d(b,y^n)
\]
exists for all $x\in M$ and, moreover, $h\in\partial \overline{M}^h$.
\end{lemma}
Given a proper geodesic metric space $(M,d)$, a horofunction $h\in\overline{M}^h$ is called a {\em Busemann point} if there exists an almost geodesic  sequence $(y^n)$ in $M$ such that $h(x) = \lim_{n\to\infty} d(x,y^n)- d(b,y^n)$ for all $x\in M$. We denote the collection of all Busemann points by $\mathcal{B}_M$.

It is known that a product metric space $(\prod_{j=1}^p M_j,d_\infty)$, where   
 \[
 d_\infty(x,y) = \max_j d_j(x_j,y_j)\mbox{\quad for $x=(x_1,\ldots,x_p),y=(y_1,\ldots,y_p)\in \prod_{j=1}^p M_j$,}
 \]
 is  a proper geodesic metric space, if each  $(M_j,d_j)$ is a  proper geodesic metric space, see for instance \cite[Proposition 2.6.6]{Pa}. 

 The horofunctions of a product  proper geodesic metric spaces have a special form, as the following theorem shows. This theorem is an extension of \cite[Proposition 8.1]{Wa4}, and the basic idea of the proof is the same. 
\begin{theorem}\label{thm:2.2} For $j=1,\ldots,p$ let $(M_j,d_j)$ be proper geodesic metric spaces. Suppose that $h$  is  a horofunction of $(\prod_{j=1}^p M_j,d_\infty)$ with basepoint $b=(b_1,\ldots,b_p)$. If $(y^n)$ is a sequence in $\prod_{j=1}^p M_j$ converging to $h$, then there exist $J\subseteq \{1,\ldots,p\}$ non-empty,  horofunctions $h_j$  in $\overline{M_j}^h$ with basepoint $b_j$ for $j\in J$, $\alpha\in\mathbb{R}^J$ with $\min_{j\in J} \alpha_j=0$, and a subsequence $(y^{n_k})$ such that  
\begin{enumerate}[(1)]
\item $d_\infty(b,y^{n_k}) -d_k(b_j,y^{m_k}_j)\to\alpha_j$ for all $j\in J$,
\item $d_\infty(b,y^{n_k}) -d_k(b_i,y^{m_k}_i)\to\infty$ for all $i\not\in J$,
\item $h_{y^{n_k}_j}\to h_j$ for all $j\in J$.
\end{enumerate}
Moreover, $h$ is of the form, 
\begin{equation}\label{horo}
h(x) = \max_{j\in J} h_j(x_j) -\alpha_j \mbox{\quad for $x=(x_1,\ldots,x_p)\in \prod_{j=1}^p M_j$.}
\end{equation}
\end{theorem} 
\begin{proof}
Let $(y^n)$ be a sequence in $\prod_{j=1}^p M_j$ such that $(h_{y^n})$ converges to a horofunction $h$. 
So $h(x)=\lim_{n\to\infty} d_\infty(x,y^n)-d_\infty(b,y^n)$ for all $x\in \prod_{j=1}^p M_j$. As the product metric space is a proper geodesic metric space, it follows from Lemma \ref{lem1} 
that $d_\infty(b,y^n)\to \infty$ as $n\to\infty$. Write $y^n :=(y_1^n,\ldots,y^n_p)$ and let  $\alpha^n_j:= d_\infty(b,y^n)-d_j(b_j,y^n_j)\geq 0$ for all $j=1,\ldots,p$ and $n\geq 0$. 

We may assume, after taking a subsequence, that $h_{y^n_j}(\cdot) := d_j(\cdot, y^n_j)-d_j(b_j,y^n_j)$ converges to $h_j\in\overline{M_j}^h$ and $\alpha^n_j\to \alpha_j\in [0,\infty]$ for all $j\in\{1,\ldots,p\}$, and $\alpha^n_{j_0}=0$ for all $n\geq 0$ for some fixed $j_0\in\{1,\ldots,p\}$.  Let $J:=\{j\colon \alpha_j<\infty\}$ and note that $j_0\in J$.
So, 
\[
h(x) = \lim_{n\to\infty} d_\infty(x,y^n)-d_\infty(b,y^n) = \lim_{n\to\infty} \max_j (d_j(x_j,y_j^n)-d_j(b_j,y_j^n) - \alpha_j^n) = \max_{j\in J} h_j(x_j)-\alpha_j.
\]
To complete the proof note that $\alpha_j<\infty$ implies that $d_j(b_j,y^n_j)\to\infty$, and hence by Lemma \ref{lem1} we find that $h_j$ is a horofunction of $(M_j,d_j)$ for $j\in J$.
\end{proof}

The following notion will be useful in the sequel. A path $\gamma\colon [0,\infty)\to (M,d)$ is a called an \emph{almost geodesic ray} if $d(\gamma(t),\gamma(0))\to\infty$,  and for each $\epsilon>0$ there exists $T\geq 0$ such that 
\[
d(\gamma(t),\gamma(s)) +d(\gamma(s),\gamma(0)) - d(\gamma(t),\gamma(0)) <\epsilon\mbox{\quad for all }t\geq s\geq T.
\]
Let $(y^n)$ be an almost geodesic sequence in a geodesic metric space $(M,d)$, and assume that 
\begin{equation}\label{eq1}
d(y^n,y^0)<d(y^{n+1},y^0)\mbox{\quad for all }n\geq 0. 
\end{equation}
For simplicity we write $\Delta_n:= d(y^{n},y^0)$ and we let  $\gamma_n\colon [0,d(y^{n+1},y^n)]\to (M,d)$ be a geodesic path connecting $y^n$ and $y^{n+1}$, i.e., $\gamma_n(0) = y^n$ and $\gamma_n(d(y^{n+1},y^n)) = y^{n+1}$. for all $n\geq 0$. 

We write $I_n:=[\Delta_n,\Delta_{n+1}]$ and let  $\bar{\gamma}_n\colon I_n\to (M,d)$ be the affine reparametrisation   of $\gamma_n$ given by 
\[
\bar{\gamma}_n(t) := \gamma_n\left(\frac{d(y^{n+1},y^n)}{ \Delta_{n+1} - \Delta_n}(t - \Delta_n)\right)\mbox{\quad for all } t \in  I_n.
\]
We call the path $\bar{\gamma}\colon [0,\infty)\to (M,d)$ given by 
\[
\bar{\gamma}(t) :=\bar{\gamma}_n(t)\mbox{\quad  for }t\in I_n
\]
a {\em ray induced by} $(y^n)$. Note that $\bar{\gamma}$ is well defined for all $t\geq 0$ by (\ref{eq1}). 

\begin{lemma}\label{lem3}
If $(y^n)$ is an almost geodesic sequence in a geodesic metric space $(M,d)$ converging to a horofunction $h$ and satisfying (\ref{eq1}), then each ray, $\bar{\gamma}$, induced by $(y^n)$ satisfies:
\begin{enumerate}[(i)]
\item $\bar{\gamma}$ is an almost geodesic ray and  $h_{\bar{\gamma}(t)}\to h$ as $t\to\infty$, 
\item the map $t\mapsto d(\bar{\gamma}(t),\bar{\gamma}(0))$ is continuous on $[0,\infty)$.
\end{enumerate}
\end{lemma}
\begin{proof}
We first show that for each $\epsilon>0$ there exists $T\geq 0$ such that 
\begin{equation}\label{eq:3.1}
d(\bar{\gamma}(t),y^n) +d(y^n,y^0) - d(\bar{\gamma}(t),y^0)<\epsilon \mbox{\quad for all $t\geq T$ and $n\geq 0$ with $t\in I_n$}.
\end{equation}
To get this inequality just note that there exists $N\geq 0$ such  that 
\begin{eqnarray*}
d(\bar{\gamma}(t),y^n) +d(y^n,y^0) - d(\bar{\gamma}(t),y^0)  & = & d(y^{n+1}, \bar{\gamma}(t))+d(\bar{\gamma}(t),y^n) +d(y^n,y^0) \\
  & & - d(\bar{\gamma}(t),y^0)-d(y^{n+1}, \bar{\gamma}(t))\\
 & \leq & d(y^{n+1},y^n) +d(y^n,y^0) - d(y^{n+1},y^0)<\epsilon, 
 \end{eqnarray*}
 for all $t\in I_n$ and $n\geq N$, as $(y^n)$ is an almost geodesic sequence. So to get (\ref{eq:3.1}) we can take $T=\Delta_N$. 

To prove that  $\bar{\gamma}$ is an almost geodesic ray, we need to show that for each $\epsilon >0$ there exists $S\geq 0$ such that 
\[
d(\bar{\gamma}(t),\bar{\gamma}(s)) +d(\bar{\gamma}(s),\bar{\gamma}(0)) - d(\bar{\gamma}(t),\bar{\gamma}(0)) <\epsilon\mbox{\quad for all }t\geq s\geq S.
\]
 Suppose that $t>s$ are such that $t\in I_n$ and $s\in I_k$ with $n>k$. Then by using (\ref{eq:3.1}) we know that for all $n$ and $k$ large, 
  \begin{eqnarray*}
  d(\bgamma(t),\bgamma(s)) +d(\bgamma(s),\bgamma(0)) - d(\bgamma(t),\bgamma(0)) & \leq & d(\bgamma(t),\bgamma(s)) +d(\bgamma(s),y^k)+ d(y^k,y^0)\\ 
   &  &\qquad  - d(\bgamma(t),y^0)\\ 
   & \leq & d(\bgamma(t),y^n)+d(y^n,\bgamma(s)) +d(\bgamma(s),y^k)\\ 
   &  & \qquad +  d(y^k,y^0) - d(\bgamma(t),y^0)\\ 
 & < & - d(y^n,y^0)+d(y^n,\bgamma(s)) +d(\bgamma(s),y^k)\\ 
   &  & \qquad+ d(y^k,y^0) +\epsilon\\ 
 & \leq & - d(y^n,y^0)+d(y^n,y^{k+1})+d(y^{k+1},\bgamma(s))\\
  &  & \qquad +d(\bgamma(s),y^k) + d(y^k,y^0) +\epsilon\\ 
&= & - d(y^n,y^0)+d(y^n,y^{k+1})+d(y^{k+1},y^k)\\ 
   &  & \qquad + d(y^k,y^0) +\epsilon\\ 
&< & - d(y^n,y^0)+d(y^n,y^{k+1})\\ 
   &  & \qquad +  d(y^{k+1},y^0) +2\epsilon <3\epsilon.\\ 
 \end{eqnarray*}
 
 Finally suppose that $t\geq s$ are such that $t,s\in I_n$. Then for all $n\geq 0$ large we have that 
   \begin{eqnarray*}
  d(\bgamma(t),\bgamma(s)) +d(\bgamma(s),\bgamma(0)) - d(\bgamma(t),\bgamma(0)) & = & d(\bgamma(t),y^n) -d(y^n,\bgamma(s))+d(\bgamma(s),\bgamma(0)) \\
   & & \qquad  - d(\bgamma(t),\bgamma(0)) \\
   &  \leq & d(\bgamma(t),y^n)+ d(y^n,y^0) - d(\bgamma(t),y^0)< \epsilon.
    \end{eqnarray*}

As $\bar{\gamma}$ is an almost geodesic ray, we know by  \cite[Lemma 4.5]{Ri} that $h_{\bar{\gamma}}(t)\to h'$, where $h'$ is a horofunction. As $\bar{\gamma}(\Delta_n)=y^n$ for all $n$, we get that $h'=h$. 

To prove the second assertion we note that the affine map 
\[
t\mapsto \frac{d(y^{n+1},y^n)}{ \Delta_{n+1} - \Delta_n}(t - \Delta_n)
\]
 is a continuous map from $I_n$ onto $[0,d(y^{n_1},y^n)]$, and the map $\gamma_n\colon [0,d(y^{n+1},y^n)]\to (M,d)$ is continuous, as $\gamma_n$ is a geodesic. Thus, the map $t\mapsto d(\bar{\gamma}(t),\bar{\gamma}(0))$ is continuous on the interior of the interval $I_n$ for all $n\geq 0$. To get continuity at the endpoints we simply note that for all $n\geq 0$,
 \[
 \lim_{t\to \Delta_n^-} d(\bgamma(t), \bgamma(0)) = d(y^n,\bgamma(0)) = \lim_{t\to \Delta_n^+} d(\bgamma(t), \bgamma(0)),
 \]
 which completes the proof.
\end{proof}
\begin{lemma}\label{lem4}
If $(y^n)$ is an almost geodesic sequence in a geodesic metric space $(M,d)$ satisfying (\ref{eq1}) and $\bar{\gamma}$ is a ray induced by $(y^n)$, then for each sequence $(\beta^n)$ in $[0,\infty)$ with $\beta^{n+1}>\beta^n$ for all $n\geq 0$ there exists sequence $(t^n)$ in $[0,\infty)$ with $t^{n+1}>t^n$ for all $n\geq 0$ such that  $d(\bgamma(t^n),\bgamma(0)) = \beta^n$ for all $n\geq 0$.
\end{lemma}
\begin{proof}
Note that as $\overline{\gamma}\colon [0,\infty)\to M$ is an almost geodesic by Lemma \ref{lem3}(i), we know that $d(\overline{\gamma}(t),\overline{\gamma}(0))\to\infty$ as $t\to\infty$.  From Lemma \ref{lem3}(ii) we know that $\overline{\gamma}$ is continuous on $[0,\infty)$, so there exists $t_0^n\geq 0$ such that  $d(\overline{\gamma}(t_0^n),\overline{\gamma}(0)) =\beta^n$.  Now if we let  $t^n:=\inf\{t\geq 0\colon d(\overline{\gamma}(t),\overline{\gamma}(0)) =\beta^n\}$, then by continuity of $\overline{\gamma}$ we have that $d(\overline{\gamma}(t^n),\overline{\gamma}(0)) =\beta^n$ and $t^{n+1}>t^n$ for all $n\geq 0$. 
\end{proof}
\subsection{Detour distance}
Suppose that $(M,d)$ is a proper geodesic metric space. Given two Busemann points $h_1,h_2\in \partial \overline{M}^h$  the {\em detour cost} is given by  
\begin{equation}\label{H}
H(h_1,h_2) := \inf_{(z^n)}\lim\inf_n d(b,z^n)  +h_2(z^n),
\end{equation}
where the infimum is taken over all sequences $(z^n)$ such that $h_{z^n}$ converges to $h_1$. It is known, see \cite[Lemma 3.1]{LW} that if $(z^n)$ is an almost geodesic  sequence converging to $h_1$ and $(w^m)$ converges to $h_2$, then 
\[
H(h_1,h_2) = \lim_{n\to\infty}\left(d(b,z^n) +\lim_{m\to\infty} d(z^n,w^m) -d(b,w^m)\right) =   \lim_{n\to\infty} d(b,z^n)  +h_2(z^n).
\]
The {\em detour distance} is given by 
\[
\delta(h_1,h_2) := H(h_1,h_2)+H(h_2,h_1).
\]
Note that for all $m,n\geq 0$ we have that 
\[
d(b,z^n) + d(z^n,w^m) -d(b,w^m)\geq 0,
\]
so that $H(h_1,h_2)\geq 0$ for all $h_1,h_2\in\partial \overline{M}^h$. It is, however, possible for  $H(h_1,h_2)$ to be infinite.  It can be shown, see \cite[Section 3]{LW}  or \cite[Section 2]{Wa3} that the detour distance is independenst of the basepoint.

The detour distance was introduced in \cite{AGW} and has been exploited and further developed in \cite{LW,Wa3}. It is known, see for instance \cite[Section 3]{LW} or \cite[Section 2]{Wa3}, that  on $\mathcal{B}_M\subseteq \partial\overline{M}^h$ the  detour distance is symmetric, satisfies the triangle inequality, and $\delta(h_1,h_2) =0$ if and only  if $h_1=h_2$. This yields a partition of $\mathcal{B}_M$ into  equivalence classes, where $h_1$ and $h_2$ are said to be equivalent if  $\delta(h_1,h_2) <\infty$. The equivalence class of $h$ will be denoted by $\mathcal{P}(h)$. Thus, the set of Busemann points, $\mathcal{B}_M$, is the disjoint union of metric spaces under the detour distance, which are called {\em parts} of $\mathcal{B}_M$. 

Isometric embeddings between proper geodesic metric spaces can be extended to the parts of the metric spaces as detour distance isometries. Indeed, suppose that $\phi\colon (M,d)\to (N,\rho)$ is an {\em isometric embedding}, i.e., $\rho(\phi(x),\phi(y)) =d(x,y)$ for all $x,y\in M$. (Note that $\phi$ need not be onto.) If $h$ is a Busemann point of $(M,d)$ with basepoint $b$ and $(z^n)$ is an almost geodesic sequence such that $(h_{z^n})$ converges to $h$, then $(u^n)$, with $u^n:=\phi(z^n)$ for $n\geq 0$, is an almost geodesic sequence in $(N,\rho)$, and hence $(h_{u^n})$ converges to a Busemann point, say $\phi(h)$, of $(N,\rho)$ with basepoint $\phi(b)$.  

We note that $\phi(h)$ is independent of the almost geodesic sequence $(z^n)$. To see this let $(w^n)$ be another almost geodesic such that $(h_{w^n})$ converges to $h$. Write $v^n:=\phi(w^n)$ for $n\geq 0$ and let $\phi(h)'$ be the limit of $(h_{v^n})$. Then 
\begin{eqnarray*}
H(h,h)  & = & \lim_{n\to\infty} d(w^n,b) + \lim_{m\to\infty} d(w^n,z^m) -d(b,z^m)\\
& = & \lim_{n\to\infty} \rho(v^n,\phi(b)) + \lim_{m\to\infty} \rho(v^n,u^m) -\rho(\phi(b),u^m)\\
 & = & H(\phi(h)',\phi(h)).
\end{eqnarray*}
Likewise, $H(\phi(h),\phi(h)') = H(h,h)$, and we deduce that $\delta(\phi(h)',\phi(h))= H(\phi(h)',\phi(h))+H(\phi(h),\phi(h)')=\delta(h,h) =0$, which shows that $\phi(h)'=\phi(h)$, as $\phi(h)'$ and $\phi(h)$ are Busemann points. 
Thus, there exists a well defined map $\Phi\colon\mathcal{B}_M\to\mathcal{B}_N$ given by $\Phi(h) :=\phi(h)$. 
\begin{lemma}\label{partiso}
If $\phi\colon (M,d)\to (N,\rho)$ is an isometric embedding, then $\Phi(\mathcal{P}(h))\subseteq \mathcal{P}(\phi(h))$ for all Busemann points $h$ of $(M,d)$ and 
\[
\delta(h',h) = \delta(\Phi(h'),\Phi(h))\mbox{\quad for all } h,h'\in\mathcal{B}_M. 
\]
\end{lemma}
\begin{proof} 
Let $(z^n)$ and $(w^n)$ be  almost geodesic sequences such that $(h_{z^n})$ converges to $h$ and $(h_{w^n})$ converges to $h'$  in $(M,d)$ with basepoint $b$. Then 
\begin{eqnarray*}
H(h',h)  & = & \lim_{n\to\infty} d(w^n,b) + \lim_{m\to\infty} d(w^n,z^m) -d(b,z^m)\\
& = & \lim_{n\to\infty} \rho(v^n,\phi(b)) + \lim_{m\to\infty} \rho(v^n,u^m) -\rho(\phi(b),u^m)\\
 & = & H(\phi(h)',\phi(h)).
\end{eqnarray*}
Likewise, $H(h,h') = H(\phi(h),\phi(h)')$, so that $\delta(h',h) = \delta(\Phi(h'),\Phi(h))$, which completes the proof. 
\end{proof}

It could happen that all parts consist of a single Busemann point, but there are also natural instances where there are nontrivial parts. In case of products of metric spaces coming from proper geodesic metric spaces, it turns out that the parts and the detour distance have a special structure that is linked to  the quotient space, $(\mathbb{R}^n/\mathrm{Sp}(\mathbf{1}), \|\cdot\|_{\mathrm{var}})$ given in (\ref{var}), as shown by the following proposition. 

 \begin{proposition}\label{prop1} If, for $j=1,\ldots, p$,  $(M_j,d_j)$ is proper geodesic metric spaces with almost geodesic sequence $(y^n_j)$ and corresponding Busemann point $h_j$ with basepoint $y^0_j$, and  $J\subseteq \{1,\ldots,p\}$ is non-empty, then the following assertions hold: 
 \begin{enumerate}[(i)] 
 \item For  $\alpha\in\mathbb{R}^J$ with $\min_{j\in J}\alpha_j=0$ the function  $h\colon (\prod_{j=1}^p M_j,d_\infty)\to \mathbb{R}$ given by, 
 \begin{equation}\label{prodbus}
 h(x) =\max_{j\in J} h_j(x_j)-\alpha_j,\mbox{\quad for $x\in\prod_{j=1}^p M_j$},
 \end{equation}
 is a Busemann point  with basepoint $y^0=(y^0_1,\ldots,y^0_p)$. Moreover, there exists an almost geodesic sequence $(z^n)$ converging to $h$, where $(z^n_j)$ is an almost geodesic converging to $h_j$ for $j\in J$ such  that  for all $n\geq 1$ we have that $d_\infty(z^n,y^0) -d_j(z^n_j,y^0_j) =\alpha_j$ for $j\in J$, and  $d_i(z^n_i,y^0_i) =0$ for all $i\not\in J$.    
\item If $\beta\in\mathbb{R}^J$ with $\min_{j\in J}\beta_j=0$ and $h'$ is a Busemann point  with basepoint $y^0$ of the form,
 \[
 h'(x) =\max_{j\in J} h_j(x_j)-\beta_j,\mbox{\quad for $x\in\prod_{j=1}^p M_j$},
 \]
then $\delta(h,h') = \|\alpha -\beta\|_{\mathrm{var}}$.  
\item For $h$ as in (\ref{prodbus}) the part $(\mathcal{P}(h),\delta)$ contains an isometric copy of $(\mathbb{R}^J/\mathrm{Sp}(\mathbf{1}), \|\cdot\|_{\mathrm{var}})$.
\end{enumerate} 
\end{proposition}
 \begin{proof}
 We know there exists an almost geodesic sequence $(y^n_j)$ in $(M_j,d_j)$ such that $h_{y^n_j}\to h_j$ as $n\to\infty$. for each $j\in J$. As $d_j(y^n_j,b_j)\to\infty$, we can take a subsequence and assume that $d_j(y^{n+1}_j,y^0_j)>d_j(y^n_j,y^0_j)>\alpha_j$ for all $n\geq 1$. Let $\bgamma_j$  be a ray induced by $(y^n_j)$. 
 
For $j\in J$ we get from  Lemma \ref{lem4}   a sequence $(t^n_j)$ in $[0,\infty)$ with $t^0_j=0$ and  
\[
d_j(\gamma_j(t^n_j),y^0_j) = (\max_{i\in J}d_i(y^n_i,y^0_i))-\alpha_j\geq 0\mbox{\quad for all }n\geq 1.
\]

Let $z^0:=(y^0_1,\ldots,y^0_p)$ and for $n\geq 1$ define $z^n=(z^n_1,\ldots,z^n_p)\in \prod_{j=1}^p M_j$ by $z^n_j:=\bgamma_j(t^n_j)$ if $j\in J$, and $z^n_j :=y^0_j$ otherwise. 

As $\min_{j\in J}\alpha_j=0$, we have for all $j\in J$, we get by construction that 
\[
d_\infty(z^n,z^0) = \max_{i\in J}d_i(y^n_i,y^0_i) = d_j(z^n_j,z^0_j) +\alpha_j\mbox{\quad for all }n\geq 1.
\]
Moreover, it follows from Lemma \ref{lem3} that $(z^n_j)$ is an almost geodesic converging to $h_j$ for $j\in J$. 

We claim that $(z^n)$ is an almost geodesic sequence in $(\prod_{j=1}^p M_j,d_\infty)$. 
Indeed, note that for $n\geq k\geq 0$ we have that 
\[
d_\infty(z^n,z^k) +d_\infty(z^k,z^0) - d_\infty(z^n,z^0) = d_j(z^n_j,z^k_j) + d_\infty(z^k,z^0) - d_\infty(z^n,z^0)
\]
for some $j=j(n,k)\in J$, as $d_j(z^n_j,z_j^k)=0$ for all $j\not\in J$. 
As $J$ is finite, we find  for all $n\geq k$ large that 
\[
d_\infty(z^n,z^k) +d_\infty(z^k,z^0) - d_\infty(z^n,z^0) =  d_j(z^n_j,z^k_j) + d_j(z^k_j,z^0_j) +\alpha_j-  d_j(z^n_j,z^0_j) -\alpha_j< \epsilon.
\]
Also for $n\geq 0$ large and $x\in \prod_{j=1}^p M_j$ we have that 
\[
h_{z^n}(x) = \max_{j\in J}( d_j(x_j,z^n_j) - d_\infty(z^n,z^0))  =  \max_{j\in J}( d_j(x_j,z^n_j) - d_j(z^n_j,z^0_j) -\alpha_j). 
\]
Letting $n\to\infty$ gives 
\[
h(x) =  \max_{j\in J}h_j(x_j) -\alpha_j\mbox{\quad for all $x\in\prod_{j=1}^p M_j$}
\]
and shows that $h$ is a Busemann point with basepoint $y^0=(y^0_1,\ldots,y^0_p)$. This completes the proof of assertion (i).

To prove the second assertion  we know from the first assertion that there exists  an almost geodesic sequence $(w^n)$ converging to $h'$, where  $(w^n_j)$  is an almost geodesic converging to $h_j$  and $d_\infty(w^n,y^0)-d_j(w_n^n,y^0_j)=\beta_j$ for $j\in J$. So, we get that 
\begin{eqnarray*}
 \max_{j\in J}\,(\beta_j-\alpha_j)& =&  \max_{j\in J}\,(H(h_j,h_j) +\beta_j-\alpha_j ) \\
& = &\max_{j\in J}\,( \lim_{n\to\infty} d_j(w^n_j,y^0_j) +\beta_j+h_j(w^n_j) -\alpha_j)\\
& = &  \max_{j\in J}\,(\lim_{n\to\infty} d_\infty(w^n,y^0) +h_j(w^n_j) -\alpha_j)\\
& = &\lim_{n\to\infty}\max_{j\in J}(d_\infty(w^n,y^0) + h_j(w^n_j) -\alpha_j)\\
& = & \lim_{n\to\infty} d_\infty(w^n,y^0) +h(w^n).
\end{eqnarray*}
Interchanging the roles of $h$ and $h'$, we find that 
\[
\delta(h',h) = H(h',h) + H(h,h') =  \max_{j\in J}(\beta_j-\alpha_j)+  \max_{j\in J} (\alpha_j-\beta_j) =\|\alpha -\beta\|_{\mathrm{var}}.
\]
The final assertion is a direct consequence of the previous  two, as $(S,\|\cdot\|_{\mathrm{var}})$ with $S:=\{\alpha\in\mathbb{R}^J\colon \min_{j\in J}\alpha_j =0\}$ is isometric to $(\mathbb{R}^J/\mathrm{Sp}(\mathbf{1}), \|\cdot\|_{\mathrm{var}})$.
 \end{proof}
 Proposition \ref{prop1} is related to \cite[Propositions  8.3 and 8.4]{Wa4}, where the Busemann points for the product of two metric spaces are characterised and the detour cost is determined.

It is interesting to understand when  a part $(\mathcal{P}(h),\delta)$ is isometric to  $(\mathbb{R}^J/\mathrm{Sp}(\mathbf{1}), \|\cdot\|_{\mathrm{var}})$.  
\begin{theorem}\label{thm:parts} 
 If, for $j=1,\ldots,q$,  $(N_j,\rho_j)$ is a proper geodesic metric space such that all horofunctions are Busemann points, and  $\delta(h_j,h_j')=\infty$ for all $h_j\neq h_j'$ Busemann points of $(N_j,\rho_j)$, then every horofunction $h$ of $(\prod_{j=1}^q N_j,d_\infty)$ is a Busemann point, and $(\mathcal{P}(h),\delta)$ is isometric to  $(\mathbb{R}^{J}/\mathrm{Sp}(\mathbf{1}),\|\cdot\|_\mathrm{var})$ for some $J\subseteq \{1,\ldots,q\}$.
 \end{theorem}
\begin{proof}
Let $h$ be a horofunction of $(\prod_{j=1}^q N_j,d_\infty)$ with respect to basepoint $b=(b_1,\ldots,b_q)$. By Theorem \ref{thm:2.2} we know that $h$ is of the form 
\[
h(x) = \max_{j\in J} h_j(x_j) -\alpha_j\mbox{\quad for $x\in \prod_{j=1}^q N_j$},
\]
and $h_j$ is a horofunction of $(N_j,\rho_j)$ with respect to basepoint $b_j$ for each $j\in J$.  
As each horofunction of $(N_j,\rho_j)$, is a Busemann point, there exists an almost geodesic sequence $(y^n_j)$ such that 
$(h_{y^n_j})$ converges to $h_j$ with basepoint $b_j$. 

For $j\not\in J$ let $y^0_j=b_j$ and define $y^0:=(y^0_1,\ldots,y^0_q)$.  Let $h^*_{j}$ be the Busemann point obtained by changing the basepoint of $h_j$ to  $y^0_j$, so $h^*_{j}(x_j) := h_{j}(x_j) - h_{j}(y^0_j)$. Now note that if we change the basepoint for $h$  to $y^0$, we get  the Busemann point 
\begin{eqnarray*}
h^*(x)  & := & h(x) -h(y^0)\\
  & = & \max_{j\in J} (h_j(x_j) -\alpha_j) - \max_{i\in J}\,( h_i(y^0_i) -\alpha_i)\\
  & = & \max_{j\in J}\,( h^*_{j}(x_j) +h_{j}(y^0_j) -\alpha_j - \max_{i\in J}\, (h_i(y^0_i) -\alpha_i))\\
   & = & \max_{j\in J} h^*_{j}(x_j) -\gamma_j,\\
\end{eqnarray*}  
where $\gamma_j : = \max_{i\in J} (h_i(y^0_i) -\alpha_i) -(h_{j}(y^0_j) -\alpha_j) \geq 0$ for $j\in J$ and $\min_{j\in J}\gamma_j =0$. 
It now follows from Proposition \ref{prop1}(i) that $h^*$ is a Busemann point  of $(\prod_{j=1}^q N_j,d_\infty)$ with respect to basepoint $y^0$, and hence $h$ is a Busemann point $(\prod_{j=1}^q N_j,d_\infty)$ with respect to basepoint $b$. Moreover, there exists an almost geodesic sequence $(z^m)$ converging to $h^*$, where $(z^m_j)$ is an almost geodesic converging to $h^*_j$ for $j\in J$, and for all $m\geq 1$ we have that $d_\infty(z^m,y^0) -d_j(z^m_j,y^0_j) =\gamma_j$ for $j\in J$, and  $d_i(z^m_i,y^0_i) =0$ for all $i\not\in J$.

To prove the second assertion we note that $(\mathcal{P}(h),\delta)$ is isometric to $(\mathcal{P}(h^*),\delta)$, since $\delta$ is independent of the basepoint.   Let $h'$ is a Busemann point  of $(\prod_{j=1}^q N_j,d_\infty)$ with respect to basepoint $y^0$ and $(w^n)$ be an almost geodesic converging to $h'$. Then by Theorem \ref{thm:2.2} we know $h'$ is of the form
 \begin{equation}\label{eq:h'}
 h'(x) =\max_{j\in J'} h'_j(x_j)-\beta_j,\mbox{\quad for $x\in\prod_{j=1}^q N_j$},
 \end{equation}
 and, after taking a subsequence, we may assume that $d_\infty(w^{n}_j,y^0) -d_k(w^{n}_j,y^0_j)\to\beta_j$ for all $j\in J'$,
$d_\infty(w^{n},y^0) -d_i(w^{n}_i,y^0_i)\to\infty$ for all $i\not\in J'$, and  $h_{w^{n}_j}\to h'_j\in\partial \overline{N_j}^h$ for all $j\in J'$.

 We claim that  if $J\neq J'$, or, $J=J'$ and $h_k\neq h_k'$ for some $k\in J$, then $\delta(h^*,h')=\infty$.
 Suppose that $J\neq J'$ and $k\in J$, but $k\not\in J'$. Then
\begin{eqnarray*}
\lim_{m\to\infty} d_\infty(w^n,z^m)-d_\infty(y^0,z^m) & = &\lim_{m\to\infty} d_\infty(w^n,z^m)-d_k(y^0_k,z^m_k) -\gamma_k\\ 
& \geq & \lim_{m\to\infty} d_k(w^n_k,z^m_k)-d_k(y^0_k,z^m_k) -\gamma_k\\ 
& \geq & -d_k(w^n_k,y^0_k) -\gamma_k,\
\end{eqnarray*}
so that 
\[
\lim_{n\to\infty} \left( d_\infty(w^n,y^0) +\lim_{m\to\infty} d_\infty(w^n,z^m)-d_\infty(y^0,z^m)\right) \geq  \lim_{n\to\infty} d_\infty(w^n,y^0)-d_k(w^n_k,y^0_k) -\gamma_k =\infty.
\]
Thus, $H(h',h^*)=\infty$ and  hence $\delta(h^*,h')=\infty$. The case where  $k\in J'$ and $k\not\in J$ can be shown in the same way.

Now suppose that $h^*_k\neq h_k'$ for some $k\in J\cap J'$. By assumption we know that $\delta(h^*_k,h'_k) =\infty$. 
Note that 
\begin{eqnarray*}
\lim_{n\to\infty}  d_\infty(w^n,y^0) +h^*(w^n) &= & \lim_{n\to\infty} d_\infty(w^n,y^0)+\max_{j\in J} h^*_j(w^n_j) -\gamma_j \\
&\geq & \liminf_{n\to\infty} d_k(w^n_k,y^0_k)+ h^*_k(w^n_k) -\gamma_k.\\
\end{eqnarray*}
It now follows from (\ref{H}) that $H(h',h^*) \geq H(h_k',h^*_k)-\gamma_k$.  Interchanging the roles of $h^*$ and $h'$ we also get that $H(h^*,h') \geq H(h^*_k,h_k')-\beta_k$, and hence $\delta(h^*,h')\geq \delta(h^*_k,h'_k) -(\gamma_k+\beta_k) =\infty$. 
 
 On the other hand, if $J=J'$ and $h^*_j=h_j'$ for all $j\in J$, then it follows from Proposition \ref{prop1}(ii) that $\delta(h^*,h') = \|\alpha -\beta\|_{\mathrm{var}}$. Moreover, it follows from that Proposition \ref{prop1}(i) that for each $\beta\in\mathbb{R}^J$ with $\min_{j\in J}\beta_j=0$ there exists a Busemann point in the part of $h^*$ of the form (\ref{eq:h'}), and hence $\mathcal{P}(h^*)$  consists of all  $h'$ of the form (\ref{eq:h'}), where $\min_{j\in J}\beta_j =0$. So if we let $S:=\{\beta\in \mathbb{R}^J\colon \min_{j\in J}\beta_j =0\}$, then $(\mathcal{P}(h^*),\delta)$ is isometric to $(S,\|\cdot\|_{\mathrm{var}})$, which in turn is isometric to the quotient space $(\mathbb{R}^{J}/\mathrm{Sp}(\mathbf{1}),\|\cdot\|_\mathrm{var})$.
 \end{proof}
An elementary example is the product space $(\mathbb{R}^n,d_\infty)$ where $d_\infty(x,y) =\max_j |x_j-y_j|$.  It is easy to verify that $(\mathbb{R},|\cdot|)$ with basepoint $0$ has only two horofunctions, namely $h_+\colon x\mapsto x$ and $h_-\colon x\mapsto -x$, both of which are Busemann points and $\delta(h_+,h_-)=\infty$. So, in this case we see that the horofunctions $h$ of $(\mathbb{R}^n,d_\infty)$ are all  Busemann points and of the form,
\[
h(x) = \max_{j\in J} \pm x_j -\alpha_j,
\]
for some $J\subseteq \{1,\ldots,n\}$ non-empty and $\alpha\in \mathbb{R}^J$ with $\min_{j\in J}\alpha_j =0$, where the sign is fixed for each $j\in J$, see also \cite[Theorem 5.2]{Guit}. Moreover, $(\mathcal{P}(h),\delta)$ is isometric to  $(\mathbb{R}^J/\mathrm{Sp}(\mathbf{1}),\|\cdot\|_\mathrm{var})$.

We are now in position to prove Theorem \ref{thm:main}.
\begin{proof}[Proof of Theorem \ref{thm:main}]
As each $(M_j,d_j)$ contains an almost geodesic sequence for $j=1,\ldots,p$, we know from Proposition \ref{prop1}(i) that  the function  $h$  of the form, $h(x) = \max_{j=1,\ldots,p} h_j(x_j)$,  is a Busemann point of $(\prod_{j=1}^p M_j,d_\infty)$. Moreover, it follows from the third part of the same proposition that $(\mathcal{P}(h),\delta)$ contains an isometric copy of $(\mathbb{R}^p/\mathrm{Sp}(\mathbf{1}),\|\cdot\|_\mathrm{var})$. 

Now suppose, for the sake of contradiction, that there exists an isometric embedding $\phi\colon (\prod_{j=1}^p M_j,d_\infty)\to (\prod_{j=1}^q N_j,d_\infty)$.  Then it follows from Lemma \ref{partiso} that the restriction of $\Phi$ to $\mathcal{P}(h)$ yields an isometric embedding of  $(\mathcal{P}(h),\delta)$ into $(\mathcal{P}(\Phi(h)),\delta')$, where $\delta'$ is the detour distance on $\mathcal{P}(\Phi(h))$.   It now follows from Theorem \ref{thm:parts}  that  $(\mathcal{P}(\Phi(h)),\delta')$ is isometric to  $(\mathbb{R}^n/\mathrm{Sp}(\mathbf{1}),\|\cdot\|_\mathrm{var})$ for some $n\in \{1,\ldots,q\}$.  So, $\Phi$ yields  an isometric embedding of $(\mathbb{R}^p/\mathrm{Sp}(\mathbf{1}),\|\cdot\|_\mathrm{var})$ into $(\mathbb{R}^n/\mathrm{Sp}(\mathbf{1}),\|\cdot\|_\mathrm{var})$ with $n<p$, which  contradicts Brouwer's invariance of domains theorem \cite{Brouw}. 
\end{proof}

\section{Product domains in $\mathbb{C}^n$}
Before we show how we can use Theorem \ref{thm:main} to derive Theorem \ref{thm:poincare}, we first recall some basic facts concerning the Kobayashi distance, see \cite[Chapter 4]{Ko} for more details.  
On the disc, $\Delta:=\{z\in\mathbb{C}\colon |z|<1\}$, the {\em hyperbolic distance}  is given by 
\[
\rho(z,w) := \log \frac{ 1+\left|\frac{w-z}{1-\bar{z}w}\right| }{1-\left|\frac{w-z}{1-\bar{z}w}\right| }=2\tanh^{-1}\left( 1 -\frac{(1-|w|^2)(1-|z|^2)}{|1-w\bar{z}|^2}\right)^{1/2}
\mbox{\quad for $z,w\in\Delta$.}\]

Given a convex domain $D\subseteq \mathbb{C}^n$, the {\em Kobayashi distance} is given by  
\[
k_D(z,w) :=\inf\{ \rho(\zeta,\eta)\colon \mbox{ $\exists f\colon \Delta \to D$ holomorphic with $f(\zeta)=z$ and $f(\eta)=w$}\}
\]
for all $z,w\in D$.  It was shown by Lempert \cite{Lem} that on bounded convex domains the Kobayashi distance coincides with  the {\em Caratheodory distance}, which is given by 
\[
c_D(z,w) := \sup_f \,\rho(f(z),f(w))
\mbox{\quad for all $z,w\in D$,}\]
where the $\sup$ is taken over all  holomorphic maps $f\colon D\to \Delta$.  

It is known, see \cite[Proposition 2.3.10]{Ab},  that if $D\subset \mathbb{C}^n$ is bounded convex domain, then $(D,k_D)$ is a proper metric space, whose topology coincides with the usual topology on $\mathbb{C}^n$. Moreover, $(D,k_D)$ is a geodesic metric space containing geodesics rays, see \cite[Theorem 2.6.19]{Ab} or \cite[Theorem 4.8.6]{Ko}.

In the case of the Euclidean ball $B^n:=\{(z_1,\ldots,z_n)\in\mathbb{C}^n\colon \|z\|^2<1\}$, where $\|z\|^2 = \sum_i |z_i|^2$, the Kobayashi distance has an explicit formula:
\[
k_{B^n}(z,w) = 2\tanh^{-1}\left( 1 -\frac{(1-\|w\|^2)(1-\|z\|^2)}{|1-\langle z,w\rangle |^2}\right)^{1/2}
\]
for all $z,w\in B^n$, see \cite[Chapters 2.2 and 2.3]{Ab}. 

On the other hand, on the polydisc $\Delta^n:=\{(z_1,\ldots,z_n)\in\mathbb{C}^n\colon \max_i |z_i|<1\}$ the Kobayashi distance satisfies
\[
k_{\Delta^n}(z,w) =\max_i \rho(z_i,w_i)\mbox{\quad for all $w=(w_1,\ldots,w_n), z=(z_1,\ldots,z_n)\in\Delta^n$,}
\]
by the product property, see \cite[Theorem 3.1.9]{Ko}.

To determine  the  horofunctions of $(B^n,k_{B^n})$, with basepoint $b=0$, it suffices to consider limits of sequences $(h_{w_n})$, where $w_n\to\xi \in\partial B^n$ in norm. 
As   
\[
k_{B^n}(z,w_n) =\log \frac{\left( |1 -\langle z,w_n\rangle | + ( |1 -\langle z,w_n\rangle |^2 - (1-\|z\|^2)(1-\|w_n\|^2))^{1/2}\right)^2}{(1-\|z\|^2)(1-\|w_n\|^2)},
\] 
and 
\[
k_{B^n}(0,w_n) = \log \frac{(1+\|w_n\|)^2}{1-\|w_n\|^2},
\]
it follows that
\begin{eqnarray*}
h(z) & = & \lim_{n\to\infty} k_{B^n}(z,w_n)-k_{B^n}(0,w_n)\\
 &  = & \log \frac{( |1 -\langle z,\xi\rangle | +|1 -\langle z,\xi\rangle |)^2}{(1-\|z\|^2)(1+\|\xi\|)^2} \\
 & = & \log\frac{ |1 -\langle z,\xi\rangle |^2}{1-\|z\|^2}.
\end{eqnarray*}
for all $z\in B^n$.  Thus, if we write 
\begin{equation}\label{hor}
h_\xi(z) :=  \log\frac{ |1 -\langle z,\xi\rangle |^2}{1-\|z\|^2}\mbox{\quad for all $z\in B^n$,}
\end{equation}
 then we obtaine $\partial \overline{B^n}^h =\{h_\xi\colon \xi\in\partial B^n\}$, see also \cite[Lemma 2.28]{Ab3} and  \cite[Remark 3.1]{KKR}.  Moreover,  each $h_\xi$ is a Busemann point, as it is the limit induced by the geodesic ray $t\mapsto \frac{e^t-1}{e^t+1}\xi$, for $0\leq t<\infty$.
 
 \begin{corollary}\label{detour1} If $h_\xi$ and $h_\eta$ are distinct horofunctions of $(B^n,k_{B^n})$, then $\delta(h_\xi,h_\eta)=\infty$.
 \end{corollary}
 \begin{proof}
 If $\xi\neq \eta$ in $\partial B^n$, then
 \[
 \lim_{z\to \eta} k_{B^n}(z,0) + h_\xi(z) =  \lim_{z\to \eta} \log\frac{1+\|z\|}{1-\|z\|} + \log\frac{ |1 -\langle z,\xi\rangle |^2}{1-\|z\|^2} =\infty,
 \]
 which implies that $\delta(h_\xi,h_\eta)=\infty$.
 \end{proof}
 Note that if $n=1$ we recover the well-known expression for the horofunctions of the hyperbolic distance on $\Delta$: 
\[
h_\xi(z) =  \log\frac{ |1 -z\overline{\xi}|^2}{1-|z|^2}=  \log\frac{ |\xi -z|^2}{1-|z|^2}\mbox{\quad for all $z\in \Delta$.}
\]

Combining (\ref{hor}) with Theorem \ref{thm:2.2} and Proposition  \ref{prop1} we get the following. 
\begin{corollary}\label{hor2} For  $B^{n_1}\times\cdots\times B^{n_q}$ the Kobayashi distance horofunctions  with basepoint $b=0$ are precisely the functions of the form, 
\[
h(z) = \max_{j\in J} \left(  \log\frac{ |1 -\langle z_j,\xi_j\rangle |^2}{1-\|z_j\|^2}-\alpha_j\right),
\]
where $J\subseteq \{1,\ldots,q\}$  non-empty,  $\xi_j\in\partial B^{n_j}$ for $j\in J$, and $\min_{j\in J} \alpha_{j}=0$.  
Moreover, each horofunction is a Busemann point, and $(\mathcal{P}(h),\delta)$ is isometric to $(\mathbb{R}^J/\mathrm{Sp}(\mathbf{1}),\|\cdot\|_\mathrm{var})$. 
\end{corollary}
Corollary \ref{hor2} should be compared  with \cite[Proposition 2.4.12]{Ab}.

A similar result holds for more general product domains. We know from \cite[Theorem 2.6.45]{Ab} that for each $\xi\in\partial D$ there exists a unique geodesic ray $\gamma_\xi\colon [0,\infty)\to D$ such that  $\gamma_\xi(0)= b$ and $\lim_{t\to\infty}\gamma_\xi(t) =\xi$, if  $D\subset \mathbb{C}$ is bounded strongly convex domain with $C^3$-boundary. We will denote the corresponding Busemann point by $h_\xi\colon D\to \mathbb{R}$, so  
\[
h_\xi(z) =\lim_{t\to\infty} k_D(z,\gamma_\xi(t)) -k_D(b,\gamma_\xi(t)).
\] 

\begin{lemma}\label{detour2} If $D\subset \mathbb{C}^n$ is a bounded strongly convex domain with $C^3$-boundary,  then each horofunction of $(D,k_D)$ is a Busemann point  and of the form $h_\xi$ for some $\xi\in\partial D$. Moreover,  $\delta(h_\xi,h_\eta) =\infty$ if  $\xi\neq \eta$. 
If $D = \prod_{i=1}^r D_i$, where each $D_i$  is a bounded strongly convex domain with $C^3$-boundary, then the horofunctions $h$ of $(D,k_D)$  are Busemann points and precisely the functions of the form, 
\[
h(z) = \max_{j\in J}  h_{\xi_j}(z_j)-\alpha_j,
\]
where $J\subseteq \{1,\ldots,r\}$  nonempty,  $\xi_j\in\partial D_j$ for $j\in J$, and $\min_{j\in J} \alpha_{j}=0$. 
\end{lemma}
 \begin{proof}
 Let $h\neq h'$ be horofunctions of $(D,k_D)$.  As $(D,k_D)$ is a proper geodesic metric space, we know there exist sequences $(w_n)$ and $(z_n)$ in $D$ such that $h_{w_n}\to h$ and $h_{z_n}\to h'$. After taking  subsequences we may assume that $w_n\to \xi\in\partial D$ and $z_n\to \eta\in\partial D$, since $D$ has a compact norm closure and $h$ and $h'$ are horofunctions.  
 
We claim that $\xi\neq \eta$. To prove this we need the assumption that $D\subset \mathbb{C}$ is a bounded strongly convex domain with $C^3$-boundary and use results by Abate \cite{Ab2} concerning the so-called small and large horospheres. These are defined as follows:  for $R>0$ the {\em small horosphere} with center $\zeta\in\partial D$ (and basepoint $b\in D$) is given by 
 \[
 \mathcal{E}(\zeta,R):=\left \{x\in D\colon \limsup_{z\to \zeta} k_D(x,z)-k_D(b,z)<\frac{1}{2}\log R\right\}
 \]  
 and the {\em large horosphere} with  center $\zeta\in\partial D$ (and basepoint $b\in D$) is given by 
 \[
 \mathcal{F}(\zeta,R):=\left \{x\in D\colon \liminf_{z\to \zeta} k_D(x,z)-k_D(b,z)<\frac{1}{2}\log R\right\}. 
 \]  
 
 We note that the horoballs
 \[\mathcal{H}(h,\frac{1}{2}\log R)=\left\{x\in D\colon  \lim_{n\to\infty} k_D(x,w_n)-k_D(b,w_n)<\frac{1}{2}\log R\right\}\]
and 
\[\mathcal{H}(h',\frac{1}{2}\log R)=\left\{x\in D\colon  \lim_{n\to\infty}k_D(x,z_n)-k_D(b,z_n)<\frac{1}{2}\log R\right\}\]
 satisfy 
 \[\mathcal{E}(\xi,R)\subseteq \mathcal{H}(h,\frac{1}{2}\log R)\subseteq  \mathcal{F}(\xi,R)\mbox{\quad and \quad }
  \mathcal{E}(\eta,R)\subseteq  \mathcal{H}(h',\frac{1}{2}\log R)\subseteq  \mathcal{F}(\eta,R).
 \]
 
 It follows from \cite[Theorem 2.6.47]{Ab} (see also \cite{Ab2}) that $ \mathcal{E}(\xi,R)= \mathcal{H}(h,\frac{1}{2}\log R)=  \mathcal{F}(\xi,R)$ and $ \mathcal{E}(\eta,R)=  \mathcal{H}(h',\frac{1}{2}\log R)=  \mathcal{F}(\eta,R)$, as $D$ is strongly convex and has $C^3$-boundary.  Thus, if $\xi=\eta$, then $h=h'$, since the horoballs, $ \mathcal{H}(h,r)$ and $ \mathcal{H}(h',r)$ for $r\in\mathbb{R}$, completely determine the horofunctions. This shows that  $\xi\neq \eta$. It follows that each horofunction is of the form $h_\xi$, with $\xi\in\partial D$, and hence a Busemann point.
 
Now suppose that $h_1$ and $h_2$ are horofunctions, with $(z^1_n)$ converging to $h_1$ and $(z^2_n)$ converging to $h_2$. After taking a subsequence we may assume that $z^1_n\to \xi_1\in\partial D$ and $z^2_n\to \xi_2\in\partial D$.  We now show that if $\xi_1\neq \xi_2$, then $h_1\neq h_2$.  This implies that  there is a one-to-one correspondence between the horofunctions of $(D,k_D)$ and $\xi\in \partial D$.

To prove  this we note that as $D$ is strongly convex, $D$ is strictly convex, i.e., for each $\nu\neq \mu$ in $\partial D$ the open straight line segment $(\nu,\mu)\subset D$.  By \cite[Theorem 1.7]{Ab2} we know that $\partial D\cap \mathrm{cl}( \mathcal{F}(h_1,R)) =\{\xi_1\}$  and $\partial D\cap \mathrm{cl}( \mathcal{F}(h_2,R)) =\{\xi_2\}$ for all $R>0$. This implies that  $\partial D\cap \mathrm{cl}( \mathcal{H}(h_1,r)) =\{\xi_1\}$  and $\partial D\cap \mathrm{cl}( \mathcal{H}(h_2,r)) =\{\xi_2\}$ for all $r\in\mathbb{R}$.  Moreover, from \cite[Lemma 5]{AR}  we know that the straight-line segment $[b,\xi_1]\subset  \mathrm{cl}( \mathcal{H}(h_1,0))$.  There  exists a neighbourhood $W\subset\mathbb{C}^n$ of $\xi_1$  such that $W\cap \mathrm{cl}( \mathcal{H}(h_2,0))=\emptyset$. If we let $w\in [b,\xi_1)\cap W$, then $h_1(w)\leq 0$, but $h_2(w)>0$, and hence $h_1\neq h_2$.

Now suppose that $\xi\neq \eta$ in $\partial D$. We know that $\partial D\cap \mathrm{cl}( \mathcal{H}(h_\eta,0)) =\{\eta\}$ and $\gamma_\xi(t)\not\in  \mathrm{cl}( \mathcal{H}(h_\eta,0))$ for all $t>0$ large. So, 
\[
H(h_\xi,h_\eta) = \lim_{t\to\infty} k_D(\gamma_\xi(t),b) + h_\eta(\gamma_\xi(t))\geq \liminf_{t\to\infty} k_D(\gamma_\xi(t),b) =\infty,
\]
since $h_\eta(\gamma_\xi(t))\geq 0$ for all $t$ large. This implies that $\delta(h_\xi,h_\eta) =\infty$. 

The final part follows directly from Theorem \ref{thm:2.2} and Proposition  \ref{prop1}. 
 \end{proof}
The proof of Theorem \ref{thm:poincare} is now elementary. 
\begin{proof}[Proof of Theorem \ref{thm:poincare}] 
If $X_j\subset \mathbb{C}^{m_j}$ is a bounded convex domain, then $(X_j,k_{X_j})$ is proper geodesic metric space which contains a geodesic ray by  \cite[Theorem 2.6.19]{Ab}. Moveover, if $Y_j\subset \mathbb{C}^{n_j}$ is a bounded strongly convex domain with $C^3$-boundary, then by Lemma \ref{detour2} all the horofunctions of $(Y_j,k_{Y_j})$ are Busemann points and any two distinct Busemann points have infinite detour distance. So, Theorem \ref{thm:main} applies and gives the desired result. 
\end{proof}
\begin{remark} 
I am grateful to  Andrew Zimmer for  sharing the following observations with me. In the case where $q=1$, Theorem \ref{thm:poincare} can be strengthened and shown in a variety of other ways. Indeed, it was shown by Balogh and Bonk \cite{BB} that the Kobayashi distance is Gromov hyperbolic on a strongly pseudo-convex domains with $C^2$-boundary, but the Kobayashi distance on a product domain is clearly not Gromov hyperbolic. This immediately implies  Theorem \ref{thm:poincare} for $q=1$  in the more general case where the image domain is strongly pseudo-convex and has $C^2$-boundary.

In fact, if $q=1$ there exists a further strengthening of Theorem \ref{thm:poincare} which only requires the image domain to be strictly convex  by using  a local argument. The isometric embedding is  a locally Lipschitz map with respect to the Euclidean norm, and hence differentiable almost everywhere by Rademacher's theorem. This implies that the embedding is also an isometric embedding under the Kobayashi infinitesimal metric. On strictly convex domains, the unit balls in the tangent spaces are strictly convex and in product domains they are not, which yields a contradiction.

Finally, for holomorphic isometric embeddings and  $q=1$, Theorem \ref{thm:poincare}  can be extended to the case where the image domain is convex with 
$C^{1,\alpha}$-boundary, see \cite[Theorem 2.22]{Zi}.
\end{remark} 

Looking at the conditions required in Theorem \ref{thm:main}, it  seems  likely  the regularity conditions on the domains $Y_j$ in Theorem \ref{thm:poincare} can be relaxed considerably. In particular one may speculate that  it suffices to assume that  each domain $Y_j$ is strictly convex and has  a $C^1$-boundary.

\end{document}